\documentclass[twoside]{amsart}
\usepackage{mathrsfs}
\usepackage{amsfonts}
\usepackage{amssymb}
\usepackage{amssymb}
\usepackage{amssymb}
\usepackage{amssymb}
\usepackage{amssymb}
\usepackage{amssymb}
\usepackage{amssymb}
\usepackage{amssymb}
\usepackage{amsmath}
\usepackage{mathdots}
\usepackage{xcolor}
\usepackage[all]{xy}

\usepackage{lastpage}

\RequirePackage{amsmath} \RequirePackage{amssymb}
\usepackage{amscd,latexsym,amsthm,amsfonts,amssymb,amsmath,amsxtra}
\usepackage[colorlinks=true, urlcolor=blue, linkcolor=red,anchorcolor=blue,citecolor=blue]{hyper ref}

    \newcommand{\BC}{{\mathbb {C}}}

     \newcommand{\CH}{{\mathcal {H}}}

     \newcommand{\CP}{{\mathcal {P}}}
     
    \newcommand{\CS}{{\mathcal {S}}}

    \newcommand{\lenth}{{\mathrm {\lenth}}}

    \newcommand{\Hom}{{\mathrm{Hom}}} 
    \newcommand{\Ind}{{\mathrm{Ind}}} \newcommand{\ind}{{\mathrm{ind}}}

    \theoremstyle{plain}

    \newtheorem{thm}{Theorem}[section] \newtheorem{corollary}[thm]{Corollary}
    \newtheorem{lemma}[thm]{Lemma}

    \numberwithin{equation}{section}

\usepackage[top=1in, bottom=1in, left=1.25in, right=1.25in]{geometry}

\begin{document}
\title{On parabolic induction and Jacquet module over p-adic group}

\begin{abstract}
In this note, using tensor products with appropriate bimodules over Hecke algebras, we uniformly describe parabolic induction and Jacquet module. We also recover a result of Loke and Przebinda on construction of big theta lift in local theta correspondence using similar description.
\end{abstract}

	\author{Jingsong Chai}
	\address{School of Mathematics, Physics and Finance \\ Anhui Polytechnic University \\Wuhu, Anhui,  241000\\China}
	\email{jingsongchai@hotmail.com}

	\subjclass[2010]{22E50}
	\keywords{Parabolic induction, Jacquet module, Local theta correspondence, Hecke algebra}
	\thanks{The author is supported by a start up funding of AHPU and NSFC 12571010.}

	\maketitle
\section{Introduction}

Constructing new representations from given old representations is very important in representation theory. Parabolic(compact) induction, Jacquet module and local theta correspondence are three widely used important methods for constructing representations over p-adic groups. In this note, we give a uniform description for these three construction methods, as tensor product over Hecke algebras with appropriate bimodules.

We begin with the case of finite groups. Let $G$ be a finite group with a subgroup $H$. If $(\sigma, W)$ is a finite dimensional representation of $H$, then we may induce $\sigma$ from $H$ to $G$ to get
\[
\Ind(\sigma):=\{ f:G\to W | f(hg)=\sigma(h)f(g) \}.
\]
This induced representation $\Ind(\sigma)$ can also be described as
\[
\Ind(\sigma)=\BC[G]\otimes_{\BC[H]} W,
\]
where $\BC[G]$ (resp. $\BC[H]$) denotes the group algebra of $G$(resp. $H$).

Now for p-adic groups, things are becoming a little more complicated. We replace the group algebra $\BC[G]$ by the Hecke algebra $\CH(G)$, which consists of smooth compactly supported functions on $G$. Now let $H$ be a closed subgroup of $G$. If $(\sigma, W)$ is a smooth representation of $H$, then we can form the compactly indcued representation $\ind_H^G(\sigma)$. By Theorem \ref{Compactly}, roughly up to modular characters, $\ind_H^G(\sigma)$ can be described as $\CH(G) \otimes_{\CH(H)} W$.

In particular, let $H=P$ be a parabolic subgroup of $G$ with Levi decomposition $P=MN$. In this case, consider the space $\CH(G/N)$, which can be viewed as a left $\CH(G)$,right $\CH(M)$ bimodule. Our first result Corollary \ref{Parabolic} states that tensoring with this bimodule over Hecke algebra $\CH(H)$ gives us the usual parabolic induction, that is, $\Ind_P^G(\sigma \otimes \delta_P^{1/2})\cong \CH(G/N)\otimes_{\CH(H)} W$.

Now if we consider $\CH(N\backslash G)$, which is a left $\CH(M)$, right $\CH(G)$ bimodule. For a smooth representation $(\pi, V)$ of $G$, we show that tensoring with this bimodule gives us the Jacquet module $J_N(\pi)$ with respect to $N$, that is, $J_N(\pi)\otimes \delta_P^{-1}\cong \CH(N\backslash G)\otimes_{\CH(G)} V$.

In \cite{LP}, Loke and Przebinda showed that the big theta in local theta correspondence can also be described as tensor product over Hecke algebras with some bimodule constructed from Weil representations. We recover their result as a consequence of Lemma \ref{identify}.

So in this way, we uniformly describe parabolic induction, Jacquet module and local theta correspondence as tensoring with some appropriate bimodules over Hecke algebras.

For tempered representations of real reductive groups, similar constructions are given by Clare in \cite{C} using $C^*$ algebras.

\section{Hecke algebra}
Let $F$ a p-adic field, and $G$ be a locally compact Hausdorff totally disconnected topological group with right Haar measure $\mu$. Let $\CH(G)$ be the space of locally constant compactly supported functions from $G$ to $\BC$. For $f_1,f_2\in \CH(G)$, we define their convolution as
\[
(f_1*f_2)(g)=\int_G f_1(gh^{-1})f_2(h)d\mu(h).
\]
This makes $\CH(G)$ to be an associative algebra, which we call Hecke algebra of $G$. If $(\pi, V)$ is a smooth representation of $G$, then we can view it as a module over $\CH(G)$ via
\[
\pi(f).v:=\int_G f(g)\pi(g)vdg,
\]
and we know this establishes an equivalence between the category of smooth representations of $G$ and the category of nondegenerate modules of $\CH(G)$.

Let $K$ be an open compact subgroup of $G$, define the function $e_K(g)$ as follows
$$
e_K(g)=\left\{\begin{array}{lll} \frac{1}{\mu(K)}, & g\in K, \\ 0, & g\notin K. \end{array} \right.
$$
Then $e_K\in \CH(G)$ and it is an idempoent element. For finitely many elements $f_1,...,f_n$ in $\CH(G)$, there exits some $e_K$ such that
$e_K*f_i=f_i*e_K=f_i$ for all $i, 1\le i\le n$.

A left module $V$ of $\CH(G)$ is called nondegenerate if $\CH(G)V=V$. Let $\CS$ be a right $\CH(G)$-module. We will define the tensor product $\CS \otimes_{\CH(G)} V$ over $\CH(G)$ as follows. First we form the tensor product $\CS \otimes V$. Let $N$ be the subspace of $\CS \otimes V$ spanned by the elements of the form
\[
sf\otimes v-s\otimes fv, \forall f\in \CH(G), s\in \CS, v\in V.
\]
We then define
\[
\CS \otimes_{\CH(G)} V:= (\CS \otimes V)/N.
\]
We use $s\otimes v$ to denote its image in $\CS \otimes_{\CH(G)} V$.

\begin{lemma}
Let $V$ be a left nondegenerate $\CH(G)$-module. Then we have a natural isomorphism
\[
\CH(G) \otimes_{\CH(G)} V \cong V.
\]
\end{lemma}
\begin{proof}
We define $\phi: \CH(G) \otimes_{\CH(G)} V \to V$ by $\phi(s\otimes v)=sv$. One check that $\phi$ is well-defined on $\CH(G) \otimes_{\CH(G)} V$. To define its inverse, for $v\in V$, there exists $f\in \CH(G), v'\in V$ such that $fv'=v$. We then define $\psi: V\to \CH(G) \otimes_{\CH(G)} V$ by $\psi(v)=f\otimes v'$. Now if there exist $f_1, f_2\in \CH(G), v_1, v_2\in V$ such that $f_1v_1=v, f_2v_2=v_2$, we want to check $f_1\otimes v_1=f_2\otimes v_2$ in $\CH(G) \otimes_{\CH(G)} V$. First note that there exists $f''\in \CH(G)$ such that $f''*f_1=f_1*f''=f_1, f''*f_2=f_2*f''=f_2$. It follows that
\[
f_1\otimes v_1=f''*f_1\otimes v_1=f''\otimes f_1v_1=f''\otimes v.
\]
Similarly, $f_2\otimes v_2=f''\otimes v$, hence $f_1\otimes v_1=f_2\otimes v_2$. So $\psi$ is well defined.

One then check that $\phi,\psi$ are inverses to each other, and thus $\CH(G) \otimes_{\CH(G)} V\cong V$.
\end{proof}

Let $H$ be another locally compact Hausdorff totally disconnected topological group. Now assume $\CS$ is also a left $\CH(H)$-module. If it satisfies $\phi(sf)=(\phi s)f$ and for all $f\in \CH(G), \phi\in \CH(H), s\in S$, then we say $\CS$ is a $\CH(H)$-$\CH(G)$-bimodule. For $s\otimes v\in \CS \otimes_{\CH(G)} V$, we can define
\[
\phi.(s\otimes v)=\phi s\otimes v,  \forall \phi\in \CH(H).
\]
One can check this is well defined and makes $\CS \otimes_{\CH(G)} V$ a left $\CH(H)$-module. Similarly, if $V$ has a $\CH(G)$-$\CH(H)$-bimodule structure, we can make $\CS \otimes_{\CH(G)} V$ to be a right $\CH(H)$-module.

\begin{lemma}
If $V_1$ is a right $\CH(H)$-module, $V_2$ is a $\CH(H)$-$\CH(G)$-bimodule, and $V_3$ is a left $\CH(G)$-module, then we have a natural isomorphism
\[
(V_1\otimes_{\CH(H)} V_2)\otimes_{\CH(G)} V_3 \cong V_1\otimes_{\CH(H)}(V_2\otimes_{\CH(G)} V_3 ).
\]
\end{lemma}

Now assume $\CS$ is a $\CH(H)$-$\CH(G)$-bimodule and $V$ is a left $\CH(H)$-module, consider $\Hom_{\CH(H)}(S, V)$. For $\phi\in \Hom_{\CH(H)}(S, V)$, $f\in \CH(G)$, we define $f\phi:S \to V$ by $(f\phi)(s)=\phi(sf)$. One can check $f\phi \in \Hom_{\CH(H)}(S, V)$, and this makes $\Hom_{\CH(H)}(S, V)$ to be a left $\CH(G)$-module.

\begin{lemma}
If $\CS$ is a $\CH(H)$-$\CH(G)$-bimodule, $V$ is a left $\CH(G)$-module, and $W$ is a left $\CH(H)$-module, then we have a natural isomorphism
\[
\Hom_{\CH(H)}(\CS\otimes_{\CH(G)} V, W)\cong \Hom_{\CH(G)} (V, \Hom_{\CH(H)}(\CS, W) ).
\]
\end{lemma}
\begin{proof}
For $\gamma \in \Hom_{\CH(H)}(\CS\otimes_{\CH(G)} V, W)$, we define $\phi(\gamma)$ by setting $\phi(\gamma)(v)(s)=\gamma (s\otimes v)$. Then one can check $\phi(\gamma) \in \Hom_{\CH(G)} (V, \Hom_{\CH(H)}(\CS, W) )$. Conversely, for $\delta\in \Hom_{\CH(G)} (V, \Hom_{\CH(H)}(\CS, W) )$, we define $\psi(\delta)(s\otimes v)=\delta (v)(s)$. One can check $\psi(\delta)$ is well defined and belongs to $\Hom_{\CH(H)}(\CS\otimes_{\CH(G)} V, W)$. Then both $\phi$ and $\psi$ are inverses to each other.
\end{proof}

As $\CS$ is a right $\CH(G)$-module, we can view it as a representation of $G$ acting on the right. Similarly, a left $\CH(G)$-module $V$ can be viewed as a representation of $G$ acting from the left. We use $N_1$ to denote the subspace of $\CS \otimes V$ spanned by elements of the form
\[
s.g\otimes v-s\otimes g.v,  \forall s\in \CS, v\in V, g\in G,
\]
and use $N_2$ to denote the subspace spanned by elements of the form
\[
s\otimes v-s.g^{-1}\otimes g.v, \forall s\in \CS, v\in V, g\in G.
\]
Put
\[
\CS \otimes_G V := (\CS \otimes V)/N_1, (\CS \otimes V)_G:=(\CS \otimes V)/N_2.
\]

\begin{lemma}
\label{identify} If both $\CS$ and $V$ are smooth representations of $G$, we have equalities
\[
\CS \otimes_{\CH(G)} V = \CS \otimes_G V = (\CS \otimes V)_G.
\]
\end{lemma}
\begin{proof}
We will show $N_1=N_2=N$. Since
\[
s.g\otimes v-s\otimes g.v=s.g\otimes v-(s.g).g^{-1}\otimes g.v
\]
and
\[
s\otimes v- s.g^{-1}\otimes g.v= (s.g^{-1}).g\otimes v -s.g^{-1}\otimes g.v,
\]
we find that $N_1=N_2$.

Now let $K$ be a compact open subgroup of $G$ fixing both $s$ and $v$. Take $f=\textbf{1}_{KgK}$ to be the characteristic function of $KgK$. Write $KgK=\coprod Kgk_i$ and $KgK=\coprod k_j'gK$ with $k_i,k_j'\in K$, and we have
\begin{eqnarray*}
&& s.\textbf{1}_{KgK} \otimes v- s\otimes \textbf{1}_{KgK}.v \\
&=& \sum_i s.\textbf{1}_{Kgk_i} \otimes v - \sum_j s\otimes \textbf{1}_{k_j'gK}.v  \\
&=& \mu(K) \sum_i s.gk_i \otimes v -  \sum_j \mu(k_j'gK) s\otimes k_j'g.v  \\
&=& \mu(K) \sum_i (s.gk_i \otimes v - s.g\otimes k_i.v+s.g\otimes k_i.v) +  \sum_j \mu(k_j'gK)(-s.k_j'\otimes g.v+ s.k_j'\otimes g.v- s\otimes k_j'g.v)  \\
&=& \mu(K) \sum_i (s.gk_i \otimes v - s.g\otimes k_i.v )+ \mu(KgK)(s.g\otimes v- s\otimes g.v ) + \sum_j \mu(k_j'gK)(s.k_j'\otimes g.v- s\otimes k_j'g.v) \\
\end{eqnarray*}
which shows that $N \subset N_1$.

For the other direction, for given $g\in G, s\in S, v\in V$, we find some compact open subgroup $K$ fixing $s$ and $gv$. Then
\begin{eqnarray*}
&& sg\otimes v- s\otimes gv \\
&=& \frac{1}{\mu(K)} (s \textbf{1}_{Kg}\otimes v -s\otimes \textbf{1}_{Kg}v +s\otimes \textbf{1}_{Kg}v )-s\otimes gv  \\
&=& \frac{1}{\mu(K)} (s \textbf{1}_{Kg}\otimes v -s\otimes \textbf{1}_{Kg}v  )
\end{eqnarray*}
which shows that $N_1 \subset N$, and thus $N=N_1$.
\end{proof}

$\bullet$ In local theta correspondence, the big theta lift is usually defined in the form $(\CS\otimes V)_G$ (see for example \cite{GKT}) with some bimodule $\CS$ from Weil representation. The above result implies that it is equal to $\CS\otimes_G V$, which recovers Theorem 3 in \cite{LP}.

\section{Compactly induced representation}
Assume $H$ is a closed subgroup of $G$, and use $\delta_H, \delta_G$ to denote the modular character of $H$ and $G$ respectively. That is, for $f\in \CH(G)$, we have
\[
\int_G f(g^{-1}x)dx = \delta_G(g)\int_G f(x)dx,
\]
and similarly for $\delta_H$.
For $f, f'\in \CH(G), \phi\in \CH(H)$, we define
\[
(f\phi)(g):=\int_H f(gh^{-1})\phi(h)dh, \ \ (f'f)(g):=(f'*f)(g).
\]
One can check this makes $\CH(G)$ to be a $\CH(G)$-$\CH(H)$-bimodule. Let $K$ be a compact open subgroup of $G$, and $\Xi_K$ be a set of representatives of double cosets $K\backslash G/ H$. For $\xi\in \Xi_K$, let $\textbf{1}_{K\xi}$ be the characteristic function of $K\xi$.

\begin{lemma}
\label{generator} If $f\in \CH(G)$ which is bi-$K$-invariant, then there exist finitely many $\xi_i \in \Xi_K, h_i\in H$, such that
\[
f=\sum_i R(h_i).\textbf{1}_{K\xi_i},
\]
where $R(h_i)$ denotes the right translation by $h_i$. In particular, $\CH(G)$ is a nondegenerate right $\CH(H)$-module.
\end{lemma}
\begin{proof}
Without loss of generality, we may assume $f=\textbf{1}_{KgK}$ for some $g\in G$. We first note that there exist finitely many $k_i \in K$, such that $KgK= \coprod_i Kgk_i$. Now for each $i$, there exists a $\xi_i \in \Xi_K$, such that $gk_i\in K\xi_i H$. This means $gk_i=k_i'\xi_ih_i$ for some $k_i'\in K, h_i\in H$. Then
\[
\textbf{1}_{Kgk_i}=\textbf{1}_{Kk_i'\xi_ih_i}=R(h_i^{-1}).\textbf{1}_{K\xi_i},
\]
and the lemma follows.
\end{proof}

For a smooth representation $(\sigma, W)$ of $H$, view it as a left $\CH(H)$-module. We consider the space of functions $f:G\to W$ satisfying the following conditions:

(1). $f(hg)=\sigma (h)f(g)$ for all $h\in H, g\in G$;

(2). $f$ is smooth;

(3). There exists a compact subset $C$ of $G$, such that $f$ is supported on $HC$.

We use $\ind_H^G(\sigma)$ to denote the space of such functions, and $G$ acts on it by right translations. We call it compactly induced representation. Take $f\otimes v \in \CH(G)\otimes_{\CH(H)} W, g\in G$, define
\[
\alpha(f\otimes w)(g):=\frac{1}{\delta_G(g)} \int_H f(g^{-1}h^{-1})\sigma(h^{-1})w\delta_H(h^{-1}) dh.
\]
One can check $\alpha (f\otimes w)\in \ind_H^G(\sigma \otimes \frac{\delta_H}{\delta_G})$, and we have the following result due to Garret (\cite{Ga}). See also Theorem III 2.6 in \cite{Re}.

\begin{thm}
\label{Compactly}
The above map $\alpha$ induces an isomorphism
\[
\alpha: \CH(G) \otimes_{\CH(H)} W \to \ind_H^G (\sigma \otimes \frac{\delta_H}{\delta_G}).
\]
\end{thm}
\begin{proof}
We first show that $\alpha$ is well defined. For $\phi\in \CH(H), f\in \CH(G), w\in W$,
\begin{eqnarray*}
\alpha(f\phi \otimes w)(g) &=&\frac{1}{\delta_G(g)} \int_H (f\phi)(g^{-1}h^{-1})\sigma(h^{-1})w\delta_H(h^{-1})dh \\
&=& \frac{1}{\delta_G(g)} \int_H \int_H f(g^{-1}h^{-1} h'^{-1})\phi(h') dh' \sigma(h^{-1})w\delta_H(h^{-1})dh
\end{eqnarray*}
and
\begin{eqnarray*}
\alpha(f\otimes \phi w)(g) &=& \frac{1}{\delta_G(g)} \int_H f(g^{-1}h^{-1})\sigma(h^{-1})(\phi w)\delta_H(h^{-1}) dh \\
&=& \frac{1}{\delta_G(g)} \int_H f(g^{-1}h^{-1}) \sigma(h^{-1}) \int_H \phi(h')\sigma(h')w dh'\delta_H(h^{-1}) dh \\
&=& \frac{1}{\delta_G(g)} \int_H \int_H f(g^{-1}h^{-1} h'^{-1})\phi(h') dh' \sigma(h^{-1})w\delta_H(h^{-1})dh
\end{eqnarray*}
which shows that $\alpha(f\phi \otimes w)=\alpha(f\otimes \phi w)$, hence $\alpha$ is well defined on $\CH(G) \otimes_{\CH(H)} W$.

Next we show $\alpha$ respects the $\CH(G)$-module structure on both sides. For $f'\in \CH(G)$, we have
\begin{eqnarray*}
(f'.\alpha(f\otimes w))(g)&=& \int_G f'(g') (g'.\alpha(f\otimes w))(g)dg' \\
&=& \int_G f'(g') \frac{1}{\delta_G(gg')} \int_H f(g'^{-1}g^{-1}h^{-1}) \sigma(h^{-1})w\delta_H(h^{-1}) dh.
\end{eqnarray*}
While
\begin{eqnarray*}
\alpha(f'f\otimes w) &=& \frac{1}{\delta_G(g)} \int_H (f'f)(g^{-1}h^{-1})\sigma(h^{-1})\delta_H(h^{-1}) dh  \\
&=& \frac{1}{\delta_G(g)} \int_H \int_G f'(g^{-1}h^{-1}g'^{-1})f(g') dg' \sigma(h^{-1})\delta_H(h^{-1}) dh  \\
&=& \frac{1}{\delta_G(g)} \int_H \int_G f'(g'^{-1}) f(g'g^{-1}h^{-1}) dg' \sigma(h^{-1})\delta_H(h^{-1}) dh \\
&=& (f'.\alpha(f\otimes w))(g)
\end{eqnarray*}
where the last equality follows from
\[
\int_G f(g'^{-1})dg'=\int_G f(g')\delta_G(g'^{-1}) dg',  \forall f\in \CH(G).
\]

Now we show that $\alpha$ is surjective. Let $K$ be a compact open subgroup of $G$, and recall that $\Xi_K$ is a set of representatives of double cosets $K\backslash G/H$. Then $\Xi_K^{-1}$ is a set of representatives of double cosets $H\backslash G/K$. For $\xi \in \Xi_K$, we use $W^{H \cap \xi^{-1} K \xi}$ to denote the vectors in $W$ fixed under $H \cap \xi^{-1} K \xi$ by the action of $\sigma\otimes \frac{\delta_H}{\delta_G}$. Choose $w\in W^{H \cap \xi K \xi^{-1}}$. Let $F\in (\ind_H^G(\sigma \otimes \frac{\delta_H}{\delta_G}))^K$ be the unique element supported on $H\xi^{-1} K$ and such that $F(\xi^{-1})=w$. Choose $f=\frac{1}{\delta_G}\frac{1}{\mu_H (H\cap \xi^{-1} K \xi)}\textbf{1}_{K\xi}$, where $\mu_H$ is the right Haar measure on $H$. Then direct computation shows that $\alpha(f\otimes w) \in (\ind_H^G(\sigma \otimes \frac{\delta_H}{\delta_G}))^K$ and $\alpha(f\otimes w)(\xi^{-1})=w$.
This implies $\alpha(f\otimes w)=F$. By Lemma 1 in section 3.5 in \cite{BH}, $\alpha$ is surjective.

Finally we show that $\alpha$ is injective. Given a nonzero element $f\otimes w \in \CH(G) \otimes_{\CH(H)} W$, assume $f$ is bi-$K$-invariant, then by Lemma \ref{generator}, $f=f=\sum_i R(h_i).\textbf{1}_{K\xi_i}$, where $\xi_i \in \Xi_K$. Hence
\[
f\otimes w=\sum_i \textbf{1}_{K\xi_i} \otimes \sigma(h_i)w=\sum_i \textbf{1}_{K\xi_i} \otimes w_i
\]
in $\CH(G)\otimes_{\CH(H)} W$.

Assume $K_i$ is a compact open subgroup of $G$ such that $K_i \cap H$ fixes $w_i$ under $\sigma\otimes \delta_H$. For each $x\in K$, there exists a compact open subgroup $K(x)\subset K$, such that $\xi_i^{-1}x^{-1} K(x) x\xi \subset K_i$. Since $\bigcup K(x)x =K$, there exists finitely many with $\bigcup K(x_j)x_j=K$. Take $K'= \cap_j K(x_j)$, and write $K=\coprod K'y$ as left cosets decomposition. Then for each $y$, $y\in K(x_j)x_j$ for some $j$, thus there exists some $k_j\in K(x_j)$ with $y=k_jx_j$. Hence
\[
\xi_i^{-1}y^{-1} K' y\xi_i \subset \xi_i^{-1}x_j^{-1}k_j^{-1} K(x_j) k_jx_j\xi_i \subset K_i.
\]

Now $\textbf{1}_{K\xi_i}=\sum_y \textbf{1}_{K'y\xi_i}$. We compute $\alpha (\textbf{1}_{K'y\xi_i} \otimes w_i)(\xi_i^{-1}y^{-1})$ as
\begin{eqnarray*}
 && \alpha (\textbf{1}_{K'y\xi_i} \otimes w_i)(\xi_i^{-1}y^{-1}) \\
 &=& \frac{1}{\delta_G(\xi_i^{-1}y^{-1})}\int_H \textbf{1}_{K'y\xi_i} (y\xi_i h^{-1})\sigma(h^{-1})w_i \delta_H(h^{-1}) dh \\
 &=& \frac{1}{\delta_G(\xi_i^{-1}y^{-1})} \mu_H(H \cap \xi_i^{-1}y^{-1}K'y\xi_i) w_i.
\end{eqnarray*}

This shows that $\alpha (\textbf{1}_{K'y\xi_i} \otimes w_i)$ is nonzero. Note that $\alpha (\textbf{1}_{K'y\xi_i} \otimes w_i)$ is supported on $H\xi_i^{-1}y^{-1}K'$, which implies that $\{\alpha (\textbf{1}_{K'y\xi_i} \otimes w_i) \}$ are linearly independent. This shows that $\alpha(f\otimes w)$ is nonzero, which implies that $\alpha$ is injective.

\end{proof}

\begin{corollary}
\label{Cor}
Let $f\in \CH(G)$. If for all $g\in G$,
\[
\int_H f(g^{-1}h^{-1})\delta_H(h^{-1})dh=0,
\]
then there exist $h_i\in H, f_i\in \CH(G)$, such that
\[
f(g) = \sum_i (f(gh_i)-f(g)), \forall g\in G.
\]
\end{corollary}
\begin{proof}
Take the trivial representation $\textbf{1}$ of $H$, and consider $\CH(G)\otimes_{\CH(H)}\textbf{1}$, then apply the above isomorphism.
\end{proof}

\section{A general construction of bimodules}
Let $X$ be a locally compact topological space. Assume $G$ acts on $X$ from the left and $H$ acts on it from the right, and these two actions are commutative.
For $f\in \CH(X), \phi\in \CH(H), \Phi\in \CH(G), x\in X$, we define
\[
(f\phi)(x):=\int_H f(x.h^{-1})\phi(h)dh, (\Phi f)(x):=\int_G \Phi(g)f(g^{-1}.x)dg.
\]
\begin{lemma}
With the above definitions, $\CH(X)$ becomes a $\CH(G)$-$\CH(H)$-bimodule.
\end{lemma}
\begin{proof}
We first check that $\CH(X)$ is a left $\CH(G)$-module.
\begin{eqnarray*}
(\Phi_1 (\Phi_2 f))(x) &=& \int_G \Phi_1(g_1) (\Phi_2 f)(g_1^{-1}.x)dg_1 \\
&=& \int_G \Phi_1(g_1) \int_G \Phi_2(g_2) f(g_2^{-1}.g_1^{-1}.x) dg_2 dg_1  \\
&=& \int_G \int_G \Phi_1 (g_1'g_2^{-1}) \Phi_2(g_2) f(g_1'^{-1}.x) dg_2 dg_1' \\
&=& ((\Phi_1 * \Phi_2)f) (x).
\end{eqnarray*}
Similarly one can check $\CH(X)$ is a right $\CH(H)$-module. Finally we check it is indeed a bimodule.
\begin{eqnarray*}
((\Phi f)\phi)(x) &=& \int_H (\Phi f)(x.h^{-1}) \phi(h) dh \\
&=& \int_H \int_G \Phi(g) f(g^{-1}.x.h^{-1}) \phi(h) dh \\
&=& \int_G \Phi(g) (f\phi)(g^{-1}.x) dg \\
&=& (\Phi (f\phi))(x).
\end{eqnarray*}
\end{proof}

\section{Parabolic induction}
Let $F$ be a non-archimedean local field, and $G$ be a reductive linear algebraic group over $F$. We will also use $G$ to denote $G(F)$ and similarly for other algebraic group over $F$. Let $P$ be a parabolic subgroup of $G$ with Levi decomposition $P=MN$, where $M$ is the Levi component of $P$ and $N$ is the unipotent radical of $P$. If $(\sigma, W)$ is a smooth representation of $M$, extend it to a smooth representation of $P$ by letting $N$ act on $W$ trivially. We still denote this smooth representation of $P$ by $(\sigma,W)$. We will form the normalized induced parabolic induction
\[
\Ind_P^G(\sigma):= \ind_P^G(\sigma\otimes \delta_P^{1/2}).
\]

Consider the space $G/N$, on which $G$ acts by left translations. For $m\in M, gN\in G/N$, define
\[
gN.m:=gmN.
\]
Then this action is well defined and commutes with the action of $G$. By results of the last section, $\CH(G/N)$ becomes a $\CH(G)$-$\CH(M)$-bimodule. As the actions of both $G$ and $M$ on $\CH(G/N)$ are smooth, $\CH(G/N)$ is nondegenerate either viewed as left $\CH(G)-$module or as right $\CH(M)$-module.

 We first have the following lemma.

\begin{lemma}
\label{Levi} Let $(\rho, \CH(P/N) )$ be the representation of $P$ given by $\rho(p)(\phi)(p'N)=\phi(p^{-1}p'N)$ for any $\phi\in \CH(P/N)$, then we have isomorphism
\[
\CH(G/N) \cong ind_P^G(\rho\otimes \delta_P).
\]
Consequently, we have identification as $\CH(G)$-$\CH(M)$-bimodules
\[
\CH(G/N) \cong \CH(G)\otimes_{\CH(P)} \CH(P/N).
\]
\end{lemma}

\begin{proof}
For any $f\in \ind_P^G(\rho\otimes \delta_P)$, define $\widetilde{f}:G/N \to \BC$ by
\[
\widetilde{f}(gN):= f(g^{-1})(N).
\]
Then one can check $\widetilde{f}$ is smooth compactly supported function on $G/N$. Note that
\begin{eqnarray*}
\widetilde{g'.f}(gN) &=& (g'.f)(g^{-1})(N) \\
&=& f(g^{-1}g')(N) \\
&=& \widetilde{f} (g'^{-1}g)(N) \\
&=&  (g'.\widetilde{f})(gN),
\end{eqnarray*}
which shows that $f\to \widetilde{f}$ is $G$-equivariant.
If $\widetilde{f}\equiv 0$, then
\[
0=\widetilde{f}(gpN)=f(p^{-1}g^{-1})(N)=\delta_P(p^{-1})f(g^{-1})(p^{-1}N).
\]
Since this is true for all $g\in G, p\in P$, we have $f\equiv 0$, and hence $f\to \widetilde{f}$ is injective.
Finally, if $K$ is a compact open subgroup and $g_0\in G$. Consider $\textbf{1}_{Kg_0^{-1}N/N} \in \CH(G/N)$. Let $f$ be the unique function in $\ind_P^G(\rho\otimes \delta_P)$ satisfying

(1).$f$ is supported on $Pg_0K$,

(2).$f$ is invariant by $K$,

(3).$f(g_0)=\delta_P^{-1}\textbf{1}_{(P\cap g_0Kg_0^{-1})N/N}$.

Then one can check $\widetilde{f}=\textbf{1}_{Kg_0^{-1}N/N} \in \CH(G/N)$. This means $f\to \widetilde{f}$ is surjective, thus $\CH(G/N)\cong ind_P^G(\rho \otimes \delta_P)$, and the result follows.
\end{proof}

\begin{corollary}
\label{Parabolic}
We have natural isomorphism
\[
\CH(G/N) \otimes_{\CH(M)} W \cong \Ind_P^G (\sigma \otimes \delta_P^{1/2}).
\]
\end{corollary}
\begin{proof}
The isomorphism follows from
\[
\CH(G/N) \otimes_{\CH(M)} W \cong (\CH(G)\otimes_{\CH(P)} \CH(P/N))\otimes_{\CH(M)} W \cong \CH(G)\otimes_{\CH(P)} W \cong \ind_P^G (\sigma \otimes \delta_P).
\]
\end{proof}

\section{Jacquet module}

We use $\widetilde{\pi}$ to denote the contragredient representation of a smooth representation $\pi$. Note that $P\backslash G$ is compact, thus $(\ind_P^G(\sigma))^\sim \cong \ind_P^G(\widetilde{\sigma}\otimes \delta_P)$ for smooth representation $\sigma$ of $M$. Then
\begin{eqnarray*}
(\CH(G)\otimes_{\CH(P)} \sigma )^{\sim}&\cong & (\ind_P^G(\sigma\otimes \delta_P))^\sim \\
&\cong & \ind_P^G (\widetilde{\sigma}) \\
&\cong & \CH(G)\otimes_{\CH(P)}(\widetilde{\sigma}\otimes \delta_P^{-1}),
\end{eqnarray*}
and consequently
\[
(\CH(G/N)\otimes_{\CH(M)} \sigma )^\sim \cong \CH(G/N) \otimes_{\CH(M)} (\widetilde{\sigma}\otimes \delta_P^{-1}).
\]

Consider the space $\CH(N\backslash G)$, for $f\in \CH(G/N)$, put $f^\vee (Ng)=f(g^{-1}N)$, then $f^\vee \in \CH(N\backslash G)$. Conversely, for any $\psi\in \CH(N\backslash G)$, put $\psi^\vee(gN)=\psi(Ng^{-1} )$, then $\psi^\vee \in \CH(G/N)$. Similarly $\CH(N\backslash G)$ carries a $\CH(M)$-$\CH(G)$-bimodule.

\begin{lemma}
\label{Adjoint}
For any smooth representations $(\pi, V)$ and $(\sigma, W)$ of $G, H$ respectively, we have natural isomorphisms
\[
\Hom_M ( \delta_P^{-1}\otimes (\CH(N\backslash G)\otimes_{\CH(G)} \pi  ),\widetilde {\sigma} ) \cong \Hom_G (\pi, (\CH(G/N)\otimes_{\CH(M)} (\sigma\otimes \delta_P^{-1} ) )^{\sim} ) \cong \Hom_G(\pi, \CH(G/N)\otimes_{\CH(M)} \widetilde{\sigma} ).
\]
\end{lemma}
\begin{proof}
The second isomorphism is immediate, so we only need to consider the first isomorphism. For any real number $r$, we will identify the representation space of $\sigma\otimes \delta_P^r$ with that of $\sigma$. Define
\[
\Phi:\Hom_M ( \delta_P^{-1}\otimes (\CH(N\backslash G)\otimes_{\CH(G)} \pi  ),\widetilde {\sigma} ) \to \Hom_G (\pi, (\CH(G/N)\otimes_{\CH(M)} (\sigma\otimes \delta_P^{-1} ) )^{\sim} )
\]
by
\[
\Phi(\alpha)(v)(f\otimes w)=\alpha(f^\vee \otimes v)(w), \forall \alpha\in \Hom_M ( \delta_P^{-1}\otimes (\CH(N\backslash G)\otimes_{\CH(G)} \pi  ),\widetilde {\sigma} ), f\in \CH(G/N), v\in V,w\in W,
\]
and define
\[
\Psi:\Hom_G (\pi, (\CH(G/N)\otimes_{\CH(M)} (\sigma\otimes \delta_P^{-1} ) )^{\sim} ) \to \Hom_M ( \delta_P^{-1}\otimes (\CH(N\backslash G)\otimes_{\CH(G)} \pi  ),\widetilde {\sigma} )
\]
by
\[
\Psi(\beta)(\psi\otimes v)(w)=\beta(v)(\psi^\vee \otimes w), \forall \beta \in \Hom_G (\pi, (\CH(G/N)\otimes_{\CH(M)} (\sigma\otimes \delta_P^{-1} ) )^{\sim} ), \psi\in \CH(N\backslash G), v\in V, w\in W.
\]

We first check $\Phi$ is well defined, that is, $\Phi(\alpha)(v)(f.m\otimes w)=\Phi(\alpha)(v)(f\otimes (\sigma(m)\delta_P^{-1}(m)w))$. We compute for $m\in M$
\begin{eqnarray*}
\Phi(\alpha)(v)(f.m\otimes w) &=& \alpha( (f.m)^\vee \otimes v  )(w) \\
&=& \alpha ( m^{-1}.f^\vee \otimes v  )(w)  \\
&=& \widetilde{\sigma}(m^{-1})\delta_P^{-1}(m) \alpha(f^\vee\otimes v)(w) \\
&=& \alpha(f^\vee\otimes v)(\sigma(m)\delta_P^{-1}(m)w ) \\
&=& \Phi(\alpha)(v)(f\otimes (\sigma(m)\delta_P^{-1}(m)w)).
\end{eqnarray*}
Next we check $\Phi(\alpha)$ is $G$-equivaraint. For any $g\in G$, we have
\begin{eqnarray*}
\Phi(\alpha)(\pi(g)v)(f\otimes w) &=& \alpha (f^\vee \otimes (\pi(g)v))(w)   \\
&=& \alpha (f^\vee.g\otimes v )(w) \\
&=& \alpha ((g^{-1}.f)^\vee \otimes v)(w)  \\
&=& \Phi(\alpha)(v)((g^{-1}.f)\otimes w )  \\
&=& \Phi(\alpha)(v)(g^{-1}.(f\otimes w) )
\end{eqnarray*}
which shows that $\Phi(\alpha)$ is $G$-equivariant.

Similarly one can show that $\Psi$ is well defined and $\Psi(\beta)$ is $M$-equivariant. Finally we can check $\Phi$ and $\Psi$ are inverses to each other, and the first isomorphism follows.

\end{proof}

Recall that $V(N)$ is the subspace of $V$ generated by elements of the form $v-\pi(u)v, \forall v\in V, u\in N$, and the Jacquet module is defined as $J_N(\pi)=V/V(N)$. The above result implies that $\CH(N\backslash G)\otimes_{\CH(G)} \pi$ is closely related to the Jacquet module $J_N(\pi)$.  The next result states that this is indeed the case.

\begin{thm}
\label{Jacquet} For any smooth representation $(\pi, V)$ of $G$, we have natural isomorphism
\[
\delta_P\otimes (\CH(N\backslash G)\otimes_{\CH(G)} \pi) \cong J_N(\pi).
\]
\end{thm}
\begin{proof}
For any $f\in \CH(G)$, define $\CP(f) \in \CH(N\backslash G)$ by
\[
\CP(f)(Ng):=\int_N f(ug) du.
\]
Then the map $\CP$ is surjective. Moreover for any $m\in M$, $\delta_P(m)(m.\CP(f))=\CP(m.f)$.

Now we define $\Phi:\delta_P\otimes (\CH(N\backslash G)\otimes_{\CH(G)} \pi) \to J_N(\pi)$ as follows. For any $\psi\in \CH(N\backslash G)$, choose $\widetilde{\psi}\in \CH(G)$ such as $\CP(\widetilde{\psi})=\psi$. Then define
\[
\Phi(\psi\otimes v):= \pi(\widetilde{\psi})v+V(N).
\]
Conversely, we define $\Psi:J_N(\pi) \to \delta_P\otimes (\CH(N\backslash G)\otimes_{\CH(G)} \pi)$ in the following way. For $v+V(N)\in J_N(\pi)$, choose an open compact subgroup $K$ of $G$, such that $v\in V^K$. Then define
\[
\Psi(v+V(N)):=\frac{1}{\mu_{G}(K)} \textbf{1}_{N\backslash NK}\otimes v,
\]
where $\mu_{G}$ denotes the Haar measure on $G$, and $\textbf{1}_{N\backslash NK}$ is the characteristic function of the open compact subset $N\backslash NK\subset N\backslash G$, which can also be identified as $\CP(\textbf{1}_K)$.

Next we check both $\Phi$ and $\Psi$ are well defined. If $\psi=0 \in \CH(N\backslash G)$, this means we choose $\widetilde{\psi}$ satisfying
\[
\CP(\widetilde{\psi})(Ng)=\int_N \widetilde{\psi}(ug)du=0
\]
for all $g\in G$. By Corollary \ref{Cor}, we can write $\widetilde{\psi}=\sum_i (\widetilde{\psi}_i- L(u_i)\widetilde{\psi}_i)$ for some $u_i\in N, \widetilde{\psi}_i\in \CH(N\backslash G)$. This will imply that
\[
\pi(\widetilde{\psi})=\sum_i (\pi(\widetilde{\psi}_i)v-\pi(u_i)\pi(\widetilde{\psi}_i)v) \in V(N).
\]
We also compute
\begin{eqnarray*}
\Phi(\psi.g\otimes v) &=& \pi(\widetilde{\psi.g})v \\
&=& \int_G \widetilde{\psi}(hg^{-1})\pi(h)v dh  \\
&=& \int_G \widetilde{\psi}(h)\pi(h)\pi(g)v dh \\
&=& \pi(\widetilde{\psi})(\pi(g)v) \\
&=& \Phi(\psi\otimes \pi(g)v),
\end{eqnarray*}
which shows that $\Phi$ is well defined.

Next we check $\Psi$ is well defined. We first note that if $K'\subset K$ is an open compact subgroup of $G$, write $K=\coprod_i K'k_i$, then
\begin{eqnarray*}
&& \frac{1}{\mu_{G}(K)}\textbf{1}_{N\backslash NK}\otimes v - \frac{1}{\mu_{G}(K')}\textbf{1}_{N\backslash NK'}\otimes v  \\
 &=& \frac{1}{\mu_{ G}(K)}\textbf{1}_{N\backslash NK}\otimes v-\frac{[K:K']}{\mu_{ G}(K)}\textbf{1}_{N\backslash NK'}\otimes v \\
 &=& \frac{1}{\mu_{G}(K)} \sum_i (\textbf{1}_{N\backslash NK'k_i}\otimes v - \textbf{1}_{N\backslash NK'}\otimes v ) \\
 &=& \frac{1}{\mu_{G}(K)} \sum_i (\textbf{1}_{N\backslash NK'}.k_i\otimes v - \textbf{1}_{N\backslash NK'}\otimes \pi(k_i)v )
\end{eqnarray*}
which shows that the definition of $\Psi$ is independent of the choice of $K$.

Next note that if $v$ is fixed by $K$, then $\pi(u)v$ is fixed by $uKu^{-1}$. We will have
\begin{eqnarray*}
&& \Psi(v-\pi(u)v)=\frac{1}{\mu_{G}(K)} \textbf{1}_{N\backslash NK}\otimes v-\frac{1}{\mu_{G}(uKu^{-1})} \textbf{1}_{N\backslash NuKu^{-1}}\otimes \pi(u)v  \\
&=& \frac{1}{\mu_{G}(K)} (\textbf{1}_{N\backslash NK}\otimes v-\textbf{1}_{N\backslash NKu^{-1}}\otimes \pi(u)v )
\end{eqnarray*}
which shows that $\Psi$ is well defined on $J_N(\pi)$.

Next we check that both $\Phi$ and $\Psi$ satisfy the claimed equivairant property. For this purpose, note that $\delta_P(m)(m.\CP(f))=\CP(m.f)$ for all $m\in M, f\in \CH(G)$, hence
\begin{eqnarray*}
\Phi( \delta_P(m)(m.(\psi\otimes v))) &=& \delta_P(m)\pi(\widetilde{m.\psi})v+V(N) \\
&=& \pi(m.\widetilde{\psi})v+V(N) \\
&=& \pi(m)\pi(\widetilde{\psi})v+V(N)
\end{eqnarray*}
which shows that $\Phi$ is $M$-equivariant. Similarly one can show $\Psi$ is also $M$-equivariant.

Finally it is routine to check $\Phi$ and $\Psi$ are inverses to each other, and the isomorphism follows.
\end{proof}



\end{document}